\documentclass[a4paper]{amsart}
\usepackage[utf8]{inputenc}

\usepackage{amsfonts,amssymb,amsmath}
\usepackage{latexsym}
\usepackage{epsfig} 
\usepackage{geometry}
\usepackage{mathrsfs}
\usepackage{url}
\usepackage{hyperref}
\usepackage{todonotes}
\usepackage[foot]{amsaddr}
\usepackage[normalem]{ulem}

\newtheorem{theorem}{Theorem}[section]

\newtheorem{definition}[theorem]{Definition}

\newtheorem{proposition}[theorem]{Proposition}
\newtheorem{assumption}[theorem]{Assumption}
\newtheorem{lemma}[theorem]{Lemma}
\newtheorem{remark}[theorem]{Remark}
\newtheorem{example}[theorem]{Example}

\numberwithin{equation}{section}

\newcommand {\real} {{\mathbb R}}
\newcommand {\E} {{\mathbb E}}

\title{Optimizing the fractional power 
in a model with stochastic PDE constraints}

\author{Carina Geldhauser}
\address{Chebyshev Laboratory, St. Petersburg State University, 14th Line V.O., 29B, Saint Petersburg 199178 Russia.}
\email{k.geldhauser@spbu.ru}

\author{Enrico Valdinoci}
\address{Department of Mathematics and Statistics, University of Western
Australia, 35 Stirling Highway, Crawley, Perth, Western
Australia 6009, Australia, and 
Dipartimento di Matematica, Universit\`a degli Studi di Milano, Via Saldini 50, 20133 Milan, Italy,
and Istituto di Matematica Applicata e Tecnologie Informatiche, Consiglio Nazionale delle Ricerche, Via Ferrata 1, 27100 Pavia, Italy}
\email{enrico@mat.uniroma3.it}

\begin{document}

\begin{abstract}
 We study an optimization problem with SPDE constraints, which has the 
peculiarity that the control parameter $s$ is the $s$-th power of the 
diffusion operator in the state equation. Well-posedness of the state 
equation and differentiability properties with respect to
the fractional 
parameter $s$ are established. We show that under certain conditions on 
the noise, optimality conditions for the control problem can be 
established. 
\end{abstract}

\maketitle

\section{Introduction}

Generally speaking, optimal control problems with  constraints are formulated as
\begin{equation}
 \min_{y \in Y, u \in U} \mathcal{J}(y,u) \quad \textup{ subject to } Constr(y,u) = 0 
\end{equation}
where $\mathcal{J}$ is a cost functional, $y$ the state variable, $u$ the control variable  and $Constr$ is a  constraint, usually in the form of an equation for $y$, called the ``state equation''. An important subcategory arises when the $Constr$
is a partial differential equation, so that the task is the identification of coefficient functions or right hand sides in the PDE: these are often called ``identification problems'' in the literature. 

The purpose of this work is to study an identification problem with two peculiarities: (1) the control variable appears as the (fractional) exponent of a diffusion operator, and (2) the constraint will be a stochastic PDE in the sense of a PDE driven by a Wiener Process. 
The optimal control parameter is therefore the answer to
the question ``\textit{what is the optimal (fractional) diffusion
pattern?}'', which appears in applications, for instance, in the following way:
In biology, $y$ represents the density of a biological species exhibiting anomalous diffusion driven by a fractional operator and combined with a random perturbation. Such fractional diffusion processes are supposed to model very well the forage behavior or certain species, see e.g. \cite{albatross}.

In engineering and economics, the problem of optimization under uncertainties is also wide-spread: Our model could serve as a very first toy problem to mathematically investigate the maximization of the probabilistic incremental Net Present Value for selecting the location of injection and production wells in petroleum engineering. The geometry and extension of such wells are crucial to the success of oil extraction in mature oil fields, where the diffusion of injected polymers within the oil field is studied, see e.g. \cite{OMVpolymerSiebererSPE, Taware2017}.

We stress that in the available literature, the term ``stochastic PDE constraints'' usually refers to deterministic PDEs with ``random input'' in the sense of random coefficients of the PDE or of the force term, see~\cite{benner2016block, quarteroni, gunzburger, manouzi, kouri2013, tiesler2012}. These problems, where the cost functional has deterministic output (due to the usage of  $(L^2(D) \otimes L^2(\Omega))$-norms, see
e.g. \cite{manouzi, rosseel2012optimal}), are interesting due to their challenges for numerical approximation, in particular avoiding the ``curse of dimensionality''.
A different approach is to study expectation and variance of random cost functionals under stochastic constraints. These arise in economics, for example when optimizing a portfolio with finitely many assets, see e.g. \cite{dentcheva2006portfolio, duncan}. The scope here is to find efficient portfolios, namely those minimizing the risk (i.e. the uncertainty of the return) or maximize the mean return for a given risk value using stochastic dominance constraints, see also \cite{markowitzbook} for an overview on the (finite dimensional) mean variance analysis.

Motivated by such applications, this work derives optimality conditions, and the existence of optimal controls for a random cost functional, a task which was, to the author's best knowledge, not considered up to now. 

The second peculiarity in our approach is that the control variable of our problem is the fractional power of the differential operator. Our work is therefore the prototypical stochastic extension of the work \cite{sprekelsvaldinoci}, where this class of identification problems was introduced for the first time in a deterministic setting. This new type of problem poses several interesting mathematical challenges, among which we mention the need for a compactness theorem adapted
for variable Banach spaces
and the need for pathwise existence of the stochastic convolution, which is crucial for the derivation of the optimal random cost functional. 

The control theory of fractional operators of diffusion type is a very new topic. Available results include the recent papers \cite{antil:14, antil:15b, antil:15, bors:15}. In these works, however, the fractional operator was fixed a priori. In our case, 
the type of fractional order operator itself is to be determined.
From the point of view of applications, it is natural to optimize over the fractional power~$s$:
As a possible application, we can interpret the model as optimizing the mean radius of search for qualified workforce around a given location (normally the company's production site). Uncertainty enters into these questions when considering non-negligible fluctuations in the mobility of the workforce, for example due to personal constraints.
We note that the use of mathematical models to deal with problems in the job market is an important topic of contemporary research, see
e.g.~\cite{econpaper2, econpaper1, PERT} and the references therein.

\textbf{Problem statement}.
Let $D \subset \real$ be a given bounded, open domain, and denote by $D_T := D \times (0,T)$ the space-time cylinder. In $D_T $,  we consider the evolution of a fractional diffusion process governed by the $s$-th power of a positive definite operator $\mathcal{L}$, which has a discrete spectrum. Note that
the fractional parameter $s>0$ can be also greater than one. The prototypical example of $\mathcal{L}$ which we have in mind is (minus) the Laplacian endowed with Dirichlet boundary conditions with domain $H^2(D) \cap H^1_0(D)$.

For a given target function $y_{D_T}(x,t) \in L^2(D_T)$,
and a non-negative smooth penalty function $\Phi (s)$, 
for each~$\omega \in \Omega$
we want to  
prove the existence of a pathwise optimal cost
$\mathcal{J}(\omega)$, defined as a minimizer in $s$ and $y$ of the cost functional
\begin{equation}\label{eq:cost}
  \mathcal{J}(y,s, \omega ) = \int_0^T \int_D |y(s,x,t,\omega) - y_{D_T}(x,t)|^2 \,dx \,dt + \Phi (s)
\end{equation}
subject to the state equation
\begin{equation}\label{eq:stateeq}
  \begin{aligned}
   dy(t) + \mathcal{L}^s y(t) dt  &=  dW(t)  \qquad \textup{ in } D \times [0,T]
    \\ y(. , 0) &= y_0 \qquad \textup{ in }  D,
  \end{aligned}
\end{equation}
where $y_0\in L^2(D)$ is a given initial condition,
and $W=W(x,t)$ a $L^2$-Wiener process.  The minimizer $\mathcal{J}(\omega)$ of \eqref{eq:cost} subject to \eqref{eq:stateeq} is called the \textit{solution to the identification problem (IP)}.

The penalty function $\Phi (s)$ is given a priori,
and, from a technical point of view,
it has to be chosen such that the problem has sufficient compactness properties in $s$. 
To this end,
to simplify technicalities, we assume that $\Phi \in C^2(0,L)$ for some~$L\in(0,+\infty]$,
that $\Phi$ is non-negative, and satisfies
\begin{equation}\label{eq:phiassumption}
 \lim_{s\to 0^+} \Phi (s) \; = \; +\infty \; = \;  \lim_{s\to L^-} \Phi (s).
\end{equation}
In \cite{sprekelsvaldinoci},
as typical examples of functions satisfying these assumptions,
the cases
$\Phi (s) = \frac{1}{s(L-s)}$, when~$L\in(0,+\infty)$,
and~$\Phi(s)=\frac{e^s}{s}$ when~$L=+\infty$ were explicitly taken into consideration.

Note that the operator $\mathcal{L}^s$ is defined as the $s$-th power of $\mathcal{L}$,
and this definition does not correspond to the usual definition of a fractional Laplacian operator via a singular integral. 
We refer e.g. to~\cite{SV-spectrum, ABA}, and to Section \ref{sec:setting} here for details about this point.

Optimizing
the fractional exponent $s$ is challenging already 
in the deterministic case, since, when the fractional parameter $s$ changes, so does the domain of definition of the operator $\mathcal{L}$, and with it the underlying space of functions of the fractional operator. This causes difficulties e.g. when proving the existence of optimal controls, as the usual compactness arguments are not directly applicable.  
Similar to the deterministic framework of  \cite{sprekelsvaldinoci}, we tackle this issue by a hand-tailored compactness argument. 

\textbf{Outline of this work}.
The structure of this work is as follows: In Section \ref{sec:existence}
we establish existence results
of solutions to \eqref{eq:stateeq} and identify the set of admissible controls. In Section \ref{sec:differentiability} we  derive the differentiability properties of the control-to-state mapping $s \mapsto u(s)$, and then use them to identify necessary and sufficient optimality conditions for the control problem (IP), which means optimizing \eqref{eq:cost} subject to \eqref{eq:stateeq}. Then, in Section \ref{sec:controls} we prove the
existence of optimal controls, namely
Theorem \ref{theo:controlexist},  and,
more specifically, we
show that $J(s, \omega)$ attains a minimum if $\omega$ is fixed,
and $s$
is in the set of admissible controls. The paper ends with an appendix
stating an ancillary result of Borel-Cantelli type, which is used in the main proofs.

We remark that the optimal
fractional parameter~$\bar{s}=\bar{s}(\omega)$ depends
on~$\omega$, since it is
obtained by the optimizing problem in~(IP) for a fixed~$\omega$
(but we often write simply~$\bar{s}$ instead
of~$\bar{s}(\omega)$
for typographical convenience).

The main results of this work are  Theorem \ref{theo:optcond}
on the optimality conditions, and Theorem \ref{theo:controlexist} on the existence of optimal controls.
Before diving into technicalities, we state a ``toy version'' of our main results, as follows:

\begin{theorem}\label{THM1.1.1.1}
Under suitable assumptions
on the regularity of the noise and on the initial data, the control problem (IP) has a solution, that is, for almost every fixed $\omega \in \Omega$, the cost functional $\mathcal{J}(\omega)$ attains a minimum in the set of admissible controls. 

Moreover, the following optimality conditions hold for a fixed realisation $\omega \in \Omega$: 
 
 \noindent \textbf{(i) necessary condition:} If $\bar{s}
=\bar{s}(\omega)$ is an optimal parameter for (IP),
and $y(\bar{s})$ the associated unique solution to the state system \eqref{eq:stateeq}, then for almost every $\omega \in \Omega$
 \begin{equation}\label{eq:necessaryintro}
  \int_0^T \int_D (y(\bar{s}) - y_D ) \partial_s y(\bar{s}) \, dx dt  \; + \; \Phi' (\bar{s}) \; = \; 0 .
 \end{equation}
  \textbf{(ii) sufficient condition:} If $\bar{s}=\bar{s}(\omega) \in (0, L)$ satisfies the necessary condition \eqref{eq:necessaryintro}, and if in addition 
   \begin{equation}
  \int_0^T \int_D \left(\partial_sy(\bar{s})\right)^2  + (y(\bar{s}) - y_D ) \partial_{ss}^2 y(\bar{s}) \, dx dt  \; + \; \Phi'' (\bar{s}) \; > \; 0 
 \end{equation}
 for almost every $\omega \in \Omega$, then $\bar{s}$ is optimal for (IP). 
\end{theorem}

The ``suitable assumptions'' are made precise in Assumptions \ref{ass:EXPL},
\ref{ass:ew} and \ref{ass:decay}, that are all stated at the beginning of Section \ref{sec:existence}. 
Roughly speaking, the assumptions are that 
the covariance operator $Q$
and the linear operator~${\mathcal{L}}$
can be diagonalized in the same basis of eigenfunctions (this is the content
of Assumption~\ref{ass:EXPL}), that
the eigenvalues of the diffusive operator are
positive and diverging (this is the content of Assumption~\ref{ass:ew}),
in fact they diverge sufficiently fast to make a fractional
series summable, and the size of the eigenvalues of $Q$
is controlled the decay of the eigenvalues of the diffusive operator
in a suitable duality sense
(a precise statement of this is given in Assumption~\ref{ass:decay}).

We would like to point out that the results we obtained are for a fixed realisation of the solution $y$, they do not imply that the solution  $\mathcal{J}(\omega, \bar{s})$ to the identification problem (IP) is a random variable, since,
due to the minimizing procedure, measurability properties may be lost. Therefore, the study of expectation and variance of our random cost functional, as in the finite dimensional case \cite{dentcheva2006portfolio, duncan, markowitzbook}, remains an open problem. 

\section{Notation and setup}\label{sec:notation}

\subsection{The functional analytic setting}\label{sec:setting}

We denote by $D\subset \real$ a bounded domain and $x\in D$ the space
variable. We will work in the space $L^2(D)$ of square-integrable
functions over $D$, and denote by $\langle .,.\rangle$  the scalar product in $L^2(D)$. 

Let $\mathcal{L}: D( \mathcal{L}) \subset L^2(D) \to L^2(D)$  be a densely defined, linear, self-adjoint, positive operator, which is not necessarily bounded but with compact inverse. 
Hence there exist  an orthonormal basis $\lbrace e_j\rbrace_{j\in \mathbb{N}}$ of $L^2(D)$ made of eigenfunctions of $\mathcal{L}$ and a sequence of real numbers $\lambda_j$ 
such that $\mathcal{L} e_j = \lambda_j e_j$ and $0 < \lambda_1 \leq \lambda_2 \leq \ldots \leq\lambda_j  \to +\infty$ as $j \to +\infty$ 
 the corresponding eigenvalues of $ \mathcal{L} $.

In this setting, we can write every function~$v\in L^2(D)$ in the form
$$v = \sum_{j=1}^{+\infty}\langle v, e_j\rangle e_j,$$ and denote 
\begin{equation}\label{eq:notationconvention}
v_j := \langle v, e_j\rangle ,
\end{equation}
so that $$v = \sum_{j=1}^{+\infty} v_j e_j.$$
The domain of $\mathcal{L}$ is characterized by
 \begin{equation}\label{eq:domain}
  \begin{aligned}
   D(\mathcal{L}) =  \left\lbrace v \in H :  \sum_{j=1}^{+\infty} \lambda_j^2 \langle v, e_j\rangle^2 < +\infty \right\rbrace .
  \end{aligned}
\end{equation}
 Thus, $- \mathcal{L}$ is the generator of an analytic semigroup of contractions which has the well-known structure  
 \begin{equation}\label{eq:semigroup}
S(t) v= \sum_{j=1}^{+\infty} e^{-\lambda_j t} {v_j e_j}. 
 \end{equation}
In our framework,
the semigroup structure will be a crucial property when we study the features of the trajectories of solutions to \eqref{eq:stateeq}, see 
the forthcoming Lemma \ref{lemma:holderpath}.
In analogy to  \cite{sprekelsvaldinoci}, we use for $v$ in the domain of $\mathcal{L}$ the notation $$
v \in \mathcal{H}^1 := \Big\{ \phi \in L^2(D) :  \lbrace \lambda_j
\langle \phi, e_j\rangle\rbrace_{j\in \mathbb{N}} \in l^2 \Big\}.$$ In this way  we can write
\begin{equation}\label{eq:defL}
  \begin{aligned}
   \mathcal{L} v &= \sum_{j=1}^{+\infty} \lambda_j \langle v, e_j\rangle e_j .
  \end{aligned}
\end{equation}
Similarly, given~$s>0$,
we can define the (spectral)
$s$-th power of  $\mathcal{L}$ via  
\begin{equation}\label{eq:defLs}
  \begin{aligned}
   \mathcal{L}^s v 
   &= \sum_{j=1}^{+\infty} \lambda^s_j \langle v, e_j\rangle e_j ,
  \end{aligned}
\end{equation}
and  describe the domain of $   \mathcal{L}^s$ as
 \begin{equation}\label{eq:sdomain}
  \begin{aligned}
   D(\mathcal{L}^s) \; = \;  \left\lbrace v = \sum_{j=1}^{+\infty} v_j e_j : v_j \in \real , 
   \textup{ with } \|v\|^2_s := \| \mathcal{L}^s v\|^2 =  \sum_{j=1}^{+\infty} \lambda_j^{2s} v_j^2 < +\infty \right\rbrace .
  \end{aligned}
\end{equation}
It is a classical approach in SPDE
to define fractional powers of linear operators in this way, see 
e.g.~\cite{dpz, lord} or also the more recent work  \cite{kruse}. 
Next, we define the space 
$$\mathcal{H}^s := \left\lbrace v \in L^2(D) : \|v\|_{\mathcal{H}^s} < 
+\infty\right\rbrace$$
with the norm
\begin{equation}\label{eq:hsnorm}
\|v\|_{\mathcal{H}^s} := \left( \sum_{j=1}^{+\infty} \lambda^{2s}_j |\langle v, e_j\rangle|^2 \right)^{1/2}.
\end{equation}

\subsection{The probabilistic setting}

The above functional analytic setting was dwelling on
properties of linear operators in Hilbert spaces. 
The classical theory of SPDEs builds upon the very same framework, 
giving conditions to make sense of solutions to PDEs with infinite-dimensional 
noise terms such as $dW(x,t)$. First of all, we recall that 
stochastic differential equations have a rigorous meaning only in their integral form, 
which for \eqref{eq:stateeq} reads
\begin{equation}\label{eq:stateeqintegral}
  y(s,x,t) = y(s,x,0) + \int_0^t \mathcal{L}^s y(s,x,\tau) d\tau + W(x,t) .
\end{equation}
 If \eqref{eq:stateeqintegral} holds $\mathbb{P}$-almost
surely, then we call these \textit{strong solutions} to SPDEs,
see \eqref{eq:strongsolution}. 
We will use a slightly different notion of solutions here
(see Definition \ref{def:solution}), which is based on the possibility to
write $y$ as an infinite sum along the orthonormal basis of the Hilbert 
space, see \eqref{eq:defSexpanded}. We discuss the two solutions concepts in Remark \ref{remark:sol}, after having laid out the necessary framework and properties.

Now we introduce the necessary notation and standard assumptions in order to write the noise $W(x,t)$ as an infinite sum of independent and identically distributed Brownian Motions.
%
%
%
We denote by $W: \Omega \times [0,T] \to  L^2(D)$ a $Q$-Wiener process with values in  $L^2(D)$. The underlying probability space is $(\Omega,  \mathcal{F}, \mathbb{P})$, and we assume that the Wiener process is adapted to a normal filtration $\mathcal{F}_t \in \mathcal{F}$.
We assume that the covariance operator $Q$ of $W$ is linear, bounded,
self-adjoint, positive semidefinite,
and that its trace is finite, namely 
\begin{equation}\label{CA5p} 
{\rm Tr}\, Q < +\infty. \end{equation}
This implies that the
sum of the eigenvalues $\mu_j$ of $Q$ is bounded. 
Note that a $Q$-Wiener process in $L^2(D)$ can be approximated in~$L^2(\Omega, C([0,T], L^2(D)))$ by a sequence of i.i.d. Brownian motions $\left\{ B_j \right\}_{j \in \mathbb{N}}$
\begin{equation}\label{eq:tracenoise}
W(x,t) =  \sum_{j=1}^{+\infty} \sqrt{\mu_j}  e_j(x) B_j(t),
\end{equation}
and by means of an exponential inequality and Borel-Cantelli Lemma, the convergence can be obtained uniformly with probability one.
Thus, the sample paths of $W(t)$ belong to $C([0,T], L^2(D))$ almost surely,  and we may therefore choose a continuous version. 

Note that  without the trace-class assumption
in~\eqref{CA5p}, the sum 
in~\eqref{eq:tracenoise} would not converge in $L^2$, but only in a larger space. In fact, due to the lack of space regularity of the noise, even
the meaning of a simple SPDE such as \eqref{eq:stateeq}  is unclear. Here
we will not dwell on weaker notions of solutions, as have been developed in recent years, because we need quite some regularity of solutions 
to ensure that the optimality conditions can be formulated in a meaningful way.

Therefore, in our discussion we will a priori restrict ourselves to trace-class noises and elliptic operators generating
analytic semigroups as in~\eqref{eq:semigroup}, which are sufficiently regularizing
in order to compete with the roughness of the noise, see Assumptions~\ref{ass:ew}
and~\ref{ass:decay} below.

Note that the smaller $s$, the less ~$\mathcal{L}$ is regularizing our solution, and the stricter assumptions we need to impose on $Q$.  The exact conditions  for the regularity of the solutions depend therefore on the interplay between $Q$ and $s$, and they
are stated in Assumption \ref{ass:decay}.

In the forthcoming analysis, especially
Subsection \ref{sec:decay}, we will make precise statements
 on the conditions on $Q$ which are necessary to ensure that our solution takes values in the space $\mathcal{H}^s$, which was defined in \eqref{eq:hsnorm}. 
We will also see that due to the influence of the noise, the set of admissible controls differs from the deterministic case. 

\section{Construction of solutions}\label{sec:existence}

Up to now, the discussion of the linear
operator $\mathcal{L}^s$ and the covariance operator $Q$
have been somehow informal, since they aimed
at expressing the main ideas without going into too technical statements.
We now make the above mentioned approach more precise. For this, as we aim to expand a solution of \eqref{eq:stateeq} along an orthonormal basis of a Hilbert space, it is convenient to assume that $Q$ has a common set of eigenfunctions with $\mathcal{L}^s$ (and so with $\mathcal{L}$), to which we have
already hinted on by using the same notation for the eigenfunctions
in~\eqref{eq:tracenoise}. 
More explicitly, we suppose the following:

\begin{assumption}\label{ass:EXPL}
For any~$j\in{\mathbb{N}}$, we have that~$Qe_j=\mu_je_j$, and~${\mathcal{L}}e_j=\lambda_j e_j$
(and thus~${\mathcal{L}}^s e_j=\lambda_j^s e_j$). In addition,
\begin{equation}\label{CA5} 
\sum_{j=1}^{+\infty}\mu_j < +\infty. \end{equation}
\end{assumption}
We remark that~\eqref{CA5} is just a restatement of~\eqref{CA5p}.
For our purposes,
it will also be technically advantageous to avoid operators
with zero or negative eigenvalues, in view of compactness
and regularity theory. Precisely, from now on we will assume the following:

\begin{assumption}\label{ass:ew}
The eigenvalues of $\mathcal{L}$ satisfy
\begin{equation}\label{eq:ewassumption}
\alpha < \lambda_1 \leq \lambda_2 \leq \ldots \lambda_j \longrightarrow 
+\infty,
\end{equation}
for some~$\alpha>0$.
\end{assumption}
The latter is a standard assumption in SPDEs, and it is satisfied
for example by  
the operators $\mathcal{L} = (- \Delta + \alpha )$
with either Neumann or Dirichlet conditions,
or~$\mathcal{L} = - \Delta$ with Dirichlet conditions. 


In the spirit of the spectral definition of the fractional 
Laplacian in \eqref{eq:defLs}, we want to find solutions of the state
equation \eqref{eq:stateeq} by approximation with real-valued
stochastic processes $ y_{j}(t,s) := \langle y(.,t) , e_j \rangle $,
where $e_j(x) \in H_0^1(D)$ is an orthonormal basis  of $L^2(D)$
built out of eigenfunctions of $\mathcal{L}$.
In other words, for fixed $s$ we define the solution of \eqref{eq:stateeq} as the infinite series
\begin{equation}\label{eq:defSexpanded}
  \begin{aligned}
   y(s)(x,t) = \sum_{j=1}^{+\infty}  \langle y(s)(x,t) , e_j(x) \rangle e_j(x) = \sum_{j=1}^{+\infty}  y_j(s,t) e_j(x).
  \end{aligned}
\end{equation}

Lemma \ref{lemma:existence} will show that this sum is convergent, i.e. \eqref{eq:defSexpanded} is well-defined, however, this is not enough to define a solution to \eqref{eq:stateeq} which exists pathwise, and this
property is in turn necessary to show the existence of optimal controls. For this, we need another assumption:


\begin{assumption}\label{ass:decay}
We assume that  $s$  is such that 
\begin{equation}\label{eq:ewdecay}
 \sum_{j=1}^{+\infty} \lambda_j^{-s} \; < \; +\infty
,\end{equation}
and that
\begin{equation}\label{eq:mudecay}
\sum_{j=1}^{+\infty} \mu_j\,\lambda_j^s\;<\;+\infty.
\end{equation}
\end{assumption}

In a sense, Assumption~\ref{ass:decay}
is a strengthening
of Assumption~\ref{ass:ew}
(which is always assumed in the following
without further mentioning it).

\begin{definition}\label{NORANOTAT}
We say that a control $s\in(0,L) $
is \textit{admissible} if it satisfies~\eqref{eq:ewdecay} and \eqref{eq:mudecay}.
We collect all such $s$ in the set $\mathscr{S}$, and call it the \textit{set of admissible controls}. Moreover we denote the  interior of $\mathscr{S}$ by $ \mathscr{S}^{\circ } $. 
\end{definition}

\begin{example}
Set $\mathcal{L} = \Delta$ on $(0,\pi)$ with Dirichlet boundary conditions.
Then the eigenfunctions read $e_j(x) :=  c_j \sin(jx)$, and the corresponding
eigenvalues are $\lambda_j = j^2$.  We get for  \eqref{eq:ewdecay}
\begin{equation}
 \sum_{j=1}^{+\infty} \lambda_j^{-s } \; =  \sum_{j=1}^{+\infty} j^{-2s },
\end{equation}
which is convergent for $s > \frac{1}{2}$. As the penalty function $\Phi$ is defined for $s \in (0,L)$ (see the problem statement, especially the paragraph around \eqref{eq:phiassumption} for details), we conclude that we
can take the set of admissible controls as the interval~$\mathscr{S} = \left( \frac{1}{2} , L \right) $.
\end{example}

With these preparatory tools at hand, we can finally state our solution
concept, and prove the existence of such solutions.

\begin{definition}\label{def:solution}
We say that $y(s): \Omega \times D \times [0,T] \to \real$ is an admissible
solution to the state equation in~\eqref{eq:stateeq} 
with initial condition~$y_0\in \mathcal{H}^{s/2}$,
if and only if the following conditions are satisfied:
\begin{enumerate}
\item The random variable $ \omega \mapsto \|y(s, \omega)\|_{L^2( [0,T] ; \mathcal{H}^s)} $ is almost surely finite for a fixed $s \in \mathscr{S}$,  
\item For fixed $s\in \mathscr{S} $, we have that
\[
 y(x,t) = \sum_{j=1}^{+\infty} y_j(t,s) e_j(x),
\]
and the stochastic processes $y_j(t,s)$ solve the It\^o diffusion equation 
\begin{equation}\label{eq:itoequation}
dy_j(t) =  - \lambda_j^s y_j(t) dt + \sqrt{ \mu_j } d B_j(t)  \qquad j \in \mathbb{N},
\end{equation}
or, in integral form,
\begin{equation}\label{eq:itoequationint}
y_j(t) =  y_j(0) -  \lambda_j^s  \int_0^t y_j(\tau) d\tau \; 
+ \; \sqrt{ \mu_j } B_j(t)  \qquad j \in \mathbb{N}
\end{equation}
\end{enumerate}
for every $t \in (0,T)$.
\end{definition}

Notice that, as $\sqrt{ \mu_j }$ and $\lambda_j^s$ are constant for fixed $j$ and $- \lambda_j^s y_j(t)$ is Lipschitz continuous,  for fixed $s$ and for every $j$, the It\^o equation \eqref{eq:itoequation} has a unique strong solution which depends continuously on the initial data, as proved for example on page~212 in \cite{LeGall}. 
We can explicitly solve \eqref{eq:itoequation} by applying It\^o's formula to $e^{\lambda_j^s t} y_j(t)$ and 
obtain
\begin{equation}\label{eq:OU}
y_j(t) = y_{j,0} e^{-\lambda_j^s t} + \sqrt{ \mu_j } \int_0^t e^{-\lambda_j^s (t - \tau)} dB_j(\tau).
\end{equation}
\textbf{Notation}: The stochastic process in~\eqref{eq:OU} consists of two parts, a deterministic mean
and a random perturbation, which is a stochastic integral. To ease the forthcoming estimates, we will abbreviate the mean part, which is a function of time depending on the parameter $s$, by $m_j(t,s)$. Moreover, the one-dimensional stochastic integral appears often as a summand in our calculations, and we abbreviate it by $ W_{\mathcal{L},s}^{j}(t)$, indicating that it is a stochastic process involving the (semigroup) of the operator $\mathcal{L}^s$. In formula, we set
\begin{equation}\label{eq:OUnotation}
m_j(t,s) := y_{j,0} e^{-\lambda_j^s t},  \qquad  W_{\mathcal{L},s}^{j}(t) := \sqrt{ \mu_j } \int_0^t e^{-\lambda_j^s (t - \tau)} dB_j(\tau) .
\end{equation}

The main part of this section will be dedicated to show that solutions in
the sense of Definition \ref{def:solution} exist. This is the content of Theorem \ref{theo:existence}, whose proof needs some preparation. 

\begin{remark}[Comparison to strong solutions]\label{remark:sol}
{\rm Our notion of solutions, namely the one in
Definition \ref{def:solution}, resembles the definition of a \textit{strong solution} to SPDEs, see \cite{dpz}, which reads for \eqref{eq:stateeq} as
\begin{equation}\label{eq:strongsolution}
 y(t) = y_0 + \int_0^t \mathcal{L}^s y(\tau) d \tau + W(t) \qquad \mathbb{P}-a.s..
\end{equation}

Strong solutions are required to be in the domain
of the differential operator,  which requires $y(\cdot ,t,s) \in
L^2(\Omega, \mathscr{H}^s)$ for any $t \in (0,T]$
and $s \in \mathscr{S}$, which will be shown in Proposition \ref{prop:l2hs}. 
Moreover, for strong solutions it is required that $\int_0^T \mathcal{L}^s y(\tau) 
d \tau < \infty$  $\mathbb{P}$-almost surely, which is the statement of
condition (1) in Definition \ref{def:solution}. 
However, we decided to propose Definition \ref{def:solution}
as our solution concept, as we need the very explicit description
of the solution as an infinite series, 
in order to be able to derive concrete optimality conditions
and the existence of optimal controls. 
In this setting, we observe that
our solutions in the sense of Definition \ref{def:solution}
are also strong solutions.
}\end{remark}

\begin{lemma}\label{lemma:existence} Let $\mathcal{L}$ satisfy 
Assumption~\ref{ass:ew} and let~$Q$ satisfy~\eqref{CA5}.
Let the initial data $y_0 \in L^2(D)$ be deterministic. Then the sum appearing  in \eqref{eq:defSexpanded} is convergent
in~$L^2(\Omega,L^2(D))$, and almost surely in~$L^2(D)$,
and its limit $y(s)(x,t)$ is a $L^2(D)$-valued adapted stochastic process.
\end{lemma}
\begin{proof}

We show first that for fixed $t$ the series in~\eqref{eq:defSexpanded}
converges in $L^2(\Omega, L^2(D))$. 
The sum \eqref{eq:defSexpanded} reads formally 
\begin{equation}\label{eq:solsum}
\begin{aligned}
y(s)(x,t) \; 
 &= \; \sum_{j=1}^{+\infty}  e_j(x)  y_{j,0}  e^{-\lambda_j^s t}  
\; + \; \sum_{j=1}^{+\infty} e_j(x) \sqrt{ \mu_j } \int_0^t e^{-\lambda_j^s (t - \tau)} dB_j(\tau) .\\
\end{aligned}
\end{equation}
Since $e^{-\lambda_j^s t} \leq 1$ for all $s,t>0$, we get for the first term
in~\eqref{eq:solsum} that 
\begin{equation}\label{t72w8h8ys92939393os}
\left\| \sum_{j=1}^{+\infty}e_j y_{j,0}\,e^{-\lambda_j^st}
\right\|^2_{L^2(D)}=\sum_{j=1}^{+\infty}|y_{j,0}|^2\,e^{-\lambda_j^st}
\le
\| y_0 \|_{L^2(D)}^2
,\end{equation}
which is finite by assumption.

To show the convergence of the second term in~\eqref{eq:solsum}, 
we recall from \eqref{eq:OUnotation} the notation $ W_{\mathcal{L},s}^{j}(t) := \sqrt{ \mu_j } \int_0^t e^{-\lambda_j^s (t - \tau)} dB_j(\tau) $,
and start by
looking at  the partial sum $\sum_{j=1}^{n} e_j(x) W_{\mathcal{L},s}^j(t)$.
As this partial sum has finitely many summands, we can exchange expectation and summation, use the one-dimensional It\^o's Isometry
and the lower bound assumption on the eigenvalues \eqref{eq:ewassumption} to obtain
\begin{equation}\label{eq:SinL2L2}
\begin{aligned}
\E \left[ \left\|\sum_{j=1}^{n} e_j(x) W_{\mathcal{L},s}^j(t) \right\|_{L^2(D)}^2 \right]
&= \; \sum_{j=1}^{n} \E  \left[
\left|   \sqrt{ \mu_j } \int_0^t e^{-\lambda_j^s (t - \tau)} dB_j(\tau)  \right|^2\right]\\
&= \;    \sum_{j=1}^{n}   \mu_j    \int_0^t e^{-2\lambda_j^s (t - \tau)} d\tau \\
&= \;     \sum_{j=1}^{n}   \frac{ \mu_j }{2\lambda_j^s} 
{\left( 1 - e^{-2\lambda_j^s t}\right) }\\&
{\leq} \; \frac{ 1 }{2}  \;  \sum_{j=1}^{n}  \frac{\mu_j}{\lambda_j^s}\\&
{\leq} \; \frac{ 1 }{2\alpha^s}  \;  \sum_{j=1}^{n}  \mu_j   ,
\end{aligned}
\end{equation}
which is finite, due to~\eqref{CA5}.
Similarly, we can calculate for $m > n$
\begin{equation}\label{eq:SinL2L2BIS}
\begin{aligned}
\E \left[ \left\|\sum_{j=1}^{m} e_j(x) W_{\mathcal{L},s}^j(t) - \sum_{j=1}^{n} e_j(x)
W_{\mathcal{L},s}^j(t) \right\|_{L^2(D)}^2\right] &= \;  \E \left[
\left\|  \sum_{l=n+1}^{m} e_l(x) W_{\mathcal{L},s}^l(t) \right\|_{L^2(D)}^2\right]\\
&= \;    \sum_{l=n+1}^{m}   \frac{ \mu_l }{2\lambda_l^s} \left( 1 - e^{-2\lambda_l^s t}\right)  \\
&\leq \;   \frac{1}{2\alpha^s}   \sum_{l=n+1}^{m}  \mu_l  ,
\end{aligned}
\end{equation}
and it follows that the sequence of partial sums $\sum_{j=1}^{n} e_j(x) W_{\mathcal{L},s}^j(t) $ is a Cauchy sequence for fixed control $s$ and time $t$. 

By~\eqref{eq:solsum}, \eqref{t72w8h8ys92939393os}, \eqref{eq:SinL2L2}
\eqref{eq:SinL2L2BIS} it follows that the 
series in~\eqref{eq:defSexpanded} is convergent
in~$L^2(\Omega,L^2(D))$. This, recalling Lemma~\ref{BORELCANTE},
gives also that this series is almost surely finite in~$L^2(D)$.
In addition, from \eqref{CA5}
and~\eqref{eq:SinL2L2} we  deduce the following boundedness in time:
\begin{equation*}
\begin{aligned}
\sup_{t\leq T} \E \left[ \left\| \sum_{j=n+1}^{+\infty} e_j(x) W_{\mathcal{L},s}^j(t)
\right\|_{L^2(D)}^2 \right]
&\le \;  \frac{ 1 }{2\alpha^s}   \sum_{l=n+1}^{+\infty}  \mu_l   \longrightarrow \; 0,
\end{aligned}
\end{equation*}
as~$n\to +\infty$,
and  so $ y(s)(.,t)$ is a $\mathcal{F}_t$-adapted $L^2(D)$-valued process.
\end{proof}

\subsection{Properties of the solution which need Assumption \ref{ass:decay}}\label{sec:decay}

As already announced in the beginning of this section, differently from the deterministic case described in \cite{sprekelsvaldinoci}, an additional assumption is necessary in the stochastic case to ensure the appropriate spatial regularity of the solution.
The following proposition will now make transparent why
Assumption \ref{ass:decay} is the appropriate
condition for our purposes:

\begin{proposition}\label{prop:l2hs}
Let the initial data $y_0 \in L^2(D)$ be deterministic,
and let Assumptions~\ref{ass:EXPL} and~\ref{ass:decay} be satisfied. 
Then the solution to the state equation  \eqref{eq:stateeq} satisfies
 \begin{equation}\label{9iwdsc87eHHHAHMA}
 y( s, t ,\cdot) \in  L^2(\Omega, \mathcal{H}^s),
\end{equation}
for any fixed~$s\in{\mathscr{S}}$ and~$t\in[0,T]$.
\end{proposition}

\begin{proof}
Recalling \eqref{eq:hsnorm} and \eqref{eq:OUnotation},
for fixed $s$, we define
$$\kappa(t) := \sup_{r>0} \left( r^2 e^{-rt}\right).$$
Then, we have that
\begin{equation}\label{eq:itoisoHsOO}
\left\|
\sum_{j=1}^{+\infty} e_j y_{j,0} \,e^{-\lambda_j^s t}
\right\|_{\mathcal{H}^s}^2
=\sum_{j=1}^{+\infty} \lambda_j^{2s}\, |y_{j,0}|^2 \,e^{-2\lambda_j^s t}
\le \kappa(t)\,\sum_{j=1}^{+\infty} |y_{j,0}|^2=\kappa(t)\,\|y_0\|^2_{L^2(D)}.\end{equation}
Also, we
exchange expectation and summation, and we apply  It\^o's Isometry to get
\begin{equation}\label{eq:itoisoHs}
\begin{split}
\E \left[ \left\|
\sum_{j=1}^N e_j W_{\mathcal{L},s}^j(t) \right\|_{{\mathcal{H}^s}}^2
\right]\;=\;&
\E \left[ \sum_{j=1}^{N} \lambda_j^{2s}  |W_{\mathcal{L},s}^j(t)|^2\right] \;
\\= \; & \sum_{j=1}^{N} \lambda_j^{2s}  \mu_j \E \left[ 
\left( \int_0^t e^{-\lambda_j^s (t - \tau)} dB_j(\tau)\right)^2  \right]\\
=  \;  &\sum_{j=1}^{N}\lambda_j^{2s}\mu_j   \int_0^t    e^{-2\lambda_j^s (t - \tau)}  d\tau \\
=  \;  &\frac12\,\sum_{j=1}^{N}\lambda_j^{s}\mu_j   \left(1 - e^{-2\lambda_j^s t }\right)\\
\leq  \; & \frac12\,\sum_{j=1}^{N}\lambda_j^{s}\mu_j
\end{split}
\end{equation}
which is finite, in light of~\eqref{eq:mudecay}
from Assumption \ref{ass:decay}.
{F}rom~\eqref{eq:itoisoHsOO}
and~\eqref{eq:itoisoHs}, we obtain~\eqref{9iwdsc87eHHHAHMA}, as desired.
\end{proof}

The next proposition deals with
the almost sure finiteness of the integral $\int_0^t \mathcal{L}^s y(s,\tau) d\tau$. 

\begin{proposition}\label{prop:l2l2hs}
Let Assumptions~\ref{ass:EXPL} and \ref{ass:decay} be satisfied. 
Let the initial data $y_0$ be deterministic,
with~$y_0 \in \mathcal{H}^{s/2}$. Then the solution to the state equation  \eqref{eq:stateeq} satisfies
 \begin{equation}\label{eq:l2ths}
 \|y(s)\|_{L^2(\Omega \times [0,T] ; \mathcal{H}^s)} \leq C
\end{equation}
for some~$C>0$. 
Moreover, for a fixed $s \in \mathscr{S}$, the random variable $ \omega \mapsto \|y(s, \omega)\|_{L^2( [0,T] ; \mathcal{H}^s)} $ is almost surely finite. 
\end{proposition}
\begin{proof} Using~\eqref{eq:OU}, we have that
\begin{eqnarray*} \|y(s)\|^2_{\mathcal{H}^s}&=&\sum_{j=1}^{+\infty} \lambda_j^{2s}\,|y_j(s)|^2\\
&=&
\sum_{j=1}^{+\infty} \lambda_j^{2s}\,\left|
y_{j,0} e^{-\lambda_j^s t} + \sqrt{ \mu_j } \int_0^t e^{-\lambda_j^s (t - \tau)} dB_j(\tau)
\right|^2.
\end{eqnarray*}
Hence, using the standard estimate~$(a+b)^2 \leq 2(a^2 + b^2)$, we conclude that
$$ \frac12\,\|y(s)\|^2_{\mathcal{H}^s}\le
\sum_{j=1}^{+\infty} \lambda_j^{2s}\,\left|
y_{j,0} e^{-\lambda_j^s t} \right|^2+\sum_{j=1}^{+\infty}\lambda_j^{2s}\,\mu_j
\left| \int_0^t e^{-\lambda_j^s (t - \tau)} dB_j(\tau)
\right|^2,$$
and therefore
\begin{equation}\label{eq:L2strongsolcondition2}
\begin{split}
\frac12\,\|y(s)\|_{L^2(\Omega \times [0,T] ; \mathcal{H}^s)}^2\,&=\frac12\,
\E \left[\int_0^T \|y(s)\|^2_{\mathcal{H}^s}\,dt\right]\\
&\le \int_0^T
\sum_{j=1}^{+\infty} \lambda_j^{2s}\,\left|
y_{j,0} e^{-\lambda_j^s t} \right|^2\,dt+
\E \left[\int_0^T\sum_{j=1}^{+\infty}\lambda_j^{2s}\,\mu_j
\left| \int_0^t e^{-\lambda_j^s (t - \tau)} dB_j(\tau)
\right|^2\,dt\right].
\end{split}\end{equation}
Now we analyze the right hand side of~\eqref{eq:L2strongsolcondition2}
term by term. First of all,
\begin{equation}\label{C5:01:01}
\begin{split}
\int_0^T
\sum_{j=1}^{+\infty} \lambda_j^{2s}\,\left|
y_{j,0} e^{-\lambda_j^s t} \right|^2\,dt\;
=\;& \int_0^T
\sum_{j=1}^{+\infty} \lambda_j^{2s}\,|y_{j,0}|^2 e^{-2\lambda_j^s t} \,dt\\
=\;& 
\sum_{j=1}^{+\infty} \lambda_j^{2s}\,|y_{j,0}|^2 \;\frac{1-e^{-2\lambda_j^s T} }{2\lambda_j^s}\\
\le\;& \frac12\,
\sum_{j=1}^{+\infty} \lambda_j^{s}\,|y_{j,0}|^2\\
=\;& \frac12\,\| y_0\|_{\mathcal{H}^{s/2}}^2.
\end{split}\end{equation}
Furthermore, by It\^o's Isometry,
$$ \E \left[
\left| \int_0^t e^{-\lambda_j^s (t - \tau)} dB_j(\tau)
\right|^2\right]
=
\E \left[
\int_0^t e^{-2\lambda_j^s (t - \tau)} d\tau\right]=\frac{1-e^{-2\lambda_j^s t}}{2\lambda_j^s}\le
\frac{1}{2\lambda_j^s},
$$
and therefore, for any fixed~$N\in{\mathbb{N}}$,
\begin{equation*}
\begin{split}
&\E \left[\int_0^T\sum_{j=1}^{N}\lambda_j^{2s}\,\mu_j
\left| \int_0^t e^{-\lambda_j^s (t - \tau)} dB_j(\tau)
\right|^2\,dt\right]\\
=\;& \int_0^T\sum_{j=1}^{N}\lambda_j^{2s}\,\mu_j\;\E \left[
\left| \int_0^t e^{-\lambda_j^s (t - \tau)} dB_j(\tau)
\right|^2\right]\,dt\\
\le\;&\frac12\int_0^T\sum_{j=1}^{N}\lambda_j^{s}\,\mu_j\,dt
\\=\;&\frac{T}2\sum_{j=1}^{N}\lambda_j^{s}\,\mu_j
\\ \le\;& c(T,s),
\end{split}\end{equation*}
for some~$c(T,s)>0$, thanks to Assumption~\ref{ass:decay}.
This and Fatou's Lemma imply that
\begin{equation}\label{C5:01:02}
\E \left[\int_0^T\sum_{j=1}^{+\infty}\lambda_j^{2s}\,\mu_j
\left| \int_0^t e^{-\lambda_j^s (t - \tau)} dB_j(\tau)
\right|^2\,dt\right]\le c(T,s).
\end{equation}
Then, combining~\eqref{eq:L2strongsolcondition2}, \eqref{C5:01:01} and~\eqref{C5:01:02},
we complete the proof of~\eqref{eq:l2ths}.

Finally, the almost sure statement in Proposition~\ref{prop:l2l2hs}
follows from~\eqref{eq:l2ths} and Lemma~\ref{BORELCANTE}.\end{proof}


\begin{theorem}\label{theo:existence}
Let Assumptions~\ref{ass:EXPL} and \ref{ass:decay} be satisfied. Let the initial data $y_0$ be deterministic,
with $y_0 \in \mathcal{H}^{s/2}$. 
Then,
for every $s\in \mathscr{S}$, there exists a unique solution $y = y(s)$ to the state system \eqref{eq:stateeq} in the sense of Definition \ref{def:solution}.
\end{theorem}  
\begin{proof}
Condition (1) in Definition \ref{def:solution} was verified in Proposition \ref{prop:l2l2hs}.
To fulfill condition (2) of Definition \ref{def:solution}, we choose a deterministic initial condition $y_{j,0} = \langle y_0(x) , e_j(x) \rangle \in \real$, and employ the series approximation of the $Q$-Wiener process \eqref{eq:tracenoise} to get the infinite system of It\^o equations
\begin{equation}\label{eq:itoequationsproof}
dy_j(t) =  - \lambda_j^s y_j(t) dt + \sqrt{ \mu_j } d B_j(t) \qquad \textup{ for } j \in \mathbb{N}.
\end{equation}
As $\sqrt{ \mu_j }$ and $\lambda_j^s$ are constant, for fixed $j$,
and $- \lambda_j^s y_j(t)$ is Lipschitz continuous, for fixed $s$, each It\^o equation  in \eqref{eq:itoequationsproof} has a unique strong solution $y_j(t,s)$, which depends continuously on the initial data, as proved for example on page~212 in \cite{LeGall}.

Lemma \ref{lemma:existence} shows that the sum $y(x,t) = \sum_{j=1}^{+\infty} y_j(t,s) e_j(x)$ is convergent,
and its limit $y(s)(x,t)$ is a $L^2(D)$-valued adapted stochastic process, which concludes the proof. 
\end{proof}

\subsection{Further space-time regularity and H\"older continuity}

In this section we prove further properties of solutions to
the state system \eqref{eq:stateeq}, which we will need in Section \ref{sec:controls}. 

\begin{proposition}\label{prop:l2l2l2-norm} 
Let $\mathcal{L}$ satisfy
Assumption~\ref{ass:ew} and let~$Q$ satisfy~\eqref{CA5}.
Let $y_0 \in L^2(D)$ be deterministic. Then any solution $y = y(s)$ to the state equation \eqref{eq:stateeq}  
satisfies the estimate
 \begin{equation}\label{eq:l2l2l2-norm-def}
 \|y(s)\|_{L^2\left(\Omega, L^2(D \times [0,T])\right)} \leq C .
\end{equation}
Moreover, for a fixed $s \in (0, 
+\infty)$, the random variable $ \omega \mapsto \|y(s, \omega)\|_{L^2(D \times [0,T]))} $ is almost surely finite. 
\end{proposition}
\begin{proof}
By~\eqref{eq:defSexpanded} and \eqref{eq:OU}, we know that
\begin{equation}\label{330}\begin{split}
\| y(s)\|^2_{L^2(D)} \;=\;& \sum_{j=1}^{+\infty} |y_j(s)|^2\\
=\;& \sum_{j=1}^{+\infty} \left|
y_{j,0}(s) e^{-\lambda_j^s t} + \sqrt{ \mu_j } \int_0^t e^{-\lambda_j^s (t - \tau)} dB_j(\tau)
\right|^2\\
\le\;& 2\left[ \sum_{j=1}^{+\infty} \left|
y_{j,0}(s) e^{-\lambda_j^s t} \right|^2+
\sum_{j=1}^{+\infty} \left| \sqrt{ \mu_j } \int_0^t e^{-\lambda_j^s (t - \tau)} dB_j(\tau)
\right|^2\right],
\end{split}\end{equation}
and therefore
 \begin{equation}\label{eq:l2l2l2-norm}
 \begin{aligned}
& \|y(s)\|_{L^2\left(\Omega, L^2(D \times [0,T])\right)}^2 \\
=\;& \E \left[ \int_0^T \| y(s)\|^2_{L^2(D)} dt \right] \\
\le\;& 2\left\{ \int_0^T \sum_{j=1}^{
+\infty}\left|y_{0,j}(s) e^{- \lambda_j^s t} \right|^2 \; +
\; \E \left[ \int_0^T \sum_{j=1}^{+\infty}  \left| \sqrt{ \mu_j }
\int_0^t e^{-\lambda_j^s (t - \tau)} dB_j(\tau) \right|^2 dt \right]\right\}.\\
 \end{aligned}
 \end{equation} 
For the first term
in the right hand side of \eqref{eq:l2l2l2-norm}
we calculate, using Assumption \ref{ass:ew}, that
 \begin{equation}\label{eq:l2l2-ic}
  \int_0^T \sum_{j=1}^{
+\infty}\left|y_{0,j}(s)  e^{- \lambda_j^s t} \right|^2  dt 
  =  \frac{1}{2} \sum_{j=1}^{+\infty}|y_{0,j}(s)|^2 \frac{1}{2\lambda_j^{s}} \; 
\left(1-e^{- 2 \lambda_j^s T}  \right)
  \leq  \frac{1}{2 \alpha^s } \| y_0\|^2_{L^{2}}.
 \end{equation}
For the second term in the right hand side of \eqref{eq:l2l2l2-norm}
we use It\^o's Isometry to get
$$
\E \left[ \left(\int_0^t e^{-\lambda_j^s (t - \tau)} dB_j(\tau) \right)^2 \right]=
\E \left[ \int_0^t e^{-2\lambda_j^s (t - \tau)} d\tau \right]=\frac{1-e^{-2\lambda_j^s t}}{2\lambda_j^s }.
$$
Hence,
we consider the partial sum  $j \leq N$, 
due to which we can exchange expectation and summation, and
conclude that
  \begin{equation}\label{eq:l2l2l2-calculation}
 \begin{aligned}
&  \E \left[ \int_0^T \sum_{j=1}^{N}
\left|   \sqrt{ \mu_j } \int_0^t e^{-\lambda_j^s (t - \tau)} dB_j(\tau) \right|^2 dt \right]\\
   =  \;& \int_0^T \sum_{j=1}^{N}   \mu_j
\E \left[ \left(\int_0^t e^{-\lambda_j^s (t - \tau)} dB_j(\tau) \right)^2 \right] dt \\
    = \;& \int_0^T  \sum_{j=1}^{N}   \frac{ \mu_j }{ 2 \lambda_j^s}
\left( 1 -  e^{-2 \lambda_j^st}\right) dt  \\
    \le \;& \int_0^T  \sum_{j=1}^{N}   \frac{ \mu_j }{ 2 \lambda_j^s}
\,dt   \\ \leq \;&c(T, s ) ,
 \end{aligned}
 \end{equation}
for some~$c(T,s)>0$,
thanks to 
Assumption~\ref{ass:ew} and the trace-class property of the noise in~\eqref{CA5}. 
Then, combining~\eqref{eq:l2l2-ic}
and~\eqref{eq:l2l2l2-calculation}, and applying Fatou's Lemma, we obtain~\eqref{eq:l2l2l2-norm-def},
as desired.

Then, from \eqref{eq:l2l2l2-norm-def} and Lemma~\ref{BORELCANTE}
we also obtain that $ \omega \mapsto \|y(s, \omega)\|_{L^2(D \times [0,T]))} $ is almost surely finite.
\end{proof}

Note that for Proposition~\ref{prop:l2l2l2-norm}  $y(s)(x,t)$ is only required to be a $L^2(D)$-valued adapted stochastic process, as proved in Lemma~\ref{lemma:existence}. The proof used only 
$L^2\left(\Omega, L^2(D \times [0,T])\right)$-norms, no additional $\mathcal{H}^s$-regularity is needed, therefore, Assumption~\ref{ass:decay} is not needed
in Proposition~\ref{prop:l2l2l2-norm}.

To ensure sufficient compactness properties needed to prove the existence of optimal controls in Section \ref{sec:controls}, we need to quantify the H\"older continuity in time of solutions to equation \eqref{eq:stateeq} in dependence of $s$. This is proved via the factorization method (see \cite{dpz}, Chapter II.5.3), which uses the semigroup generated by $\mathcal{L}^s$
and works with interpolation spaces. 

\begin{lemma}\label{lemma:holderpath}
Let the initial data $y_0 \in L^2(D)$ be deterministic, and let Assumptions~\ref{ass:EXPL} and \ref{ass:decay} be satisfied.  Then the sample paths of the process $y(s)(x,t)$ are in $C^{\delta}([0,T], L^2(D))$ for arbitrary $\delta \in (0, \frac{1}{2})$.
\end{lemma}
\begin{proof}
It suffices to verify that the trajectories  of the stochastic convolution are $\delta$-H\"older continuous. According to  \cite{dpz}, Theorem~5.15, this holds with $\delta \in (0, \beta - \epsilon )$ if the following condition on the Hilbert-Schmidt norm of $S(t)Q$, where $S(t)$ is the semigroup generated by $\mathcal{L}^s$ (see \eqref{eq:semigroup}), is satisfied:
\begin{equation}\label{eq:semigroupcondition}
 \int_0^T t^{- 2 \beta} \| S(t) Q \|^2_{HS} \;dt < +\infty,
\end{equation}
where
$$ \| P \|^2_{HS}:=\sqrt{\sum_{j=1}^{+\infty} \|Pe_j\|^2_{L^2(D)} }.$$
We observe that, in light of Assumption~\ref{ass:EXPL}
and~\eqref{eq:semigroup},
\begin{eqnarray*}
\| S(t) Q \|^2_{HS} &=& \sum_{j=1}^{+\infty} \| S(t) Q e_j\|^2_{L^2(D)}
\\&=& \sum_{j=1}^{+\infty} \mu_j^2\| S(t) e_j\|^2_{L^2(D)}
\\&=& \sum_{j=1}^{+\infty} \mu_j^2\| e^{-\lambda_j t} e_j\|^2_{L^2(D)}
\\&=& \sum_{j=1}^{+\infty} \mu_j^2 e^{-2\lambda_j t} \,.
\end{eqnarray*}
Notice also that $\mu_j$ is a bounded sequence, due to~\eqref{CA5},
therefore
$$ \| S(t) Q \|^2_{HS}\le \sum_{j=1}^{+\infty} \mu_j^2 \le
\sup_{j\in{\mathbb{N}}}\mu_j\,\sum_{j=1}^{+\infty}\mu_j,$$
which is finite, again by virtue of~\eqref{CA5}.

This gives
that~\eqref{eq:semigroupcondition} is verified when $\beta < \frac{1}{2}$,
which proves the desired result.
\end{proof}

\section{Differentiability of the control-to-state operator}\label{sec:differentiability}

In this section we prepare the way to formulate the optimality conditions. For this, it is necessary to look at the partial derivative of solutions to \eqref{eq:stateeq} in the control variable $s$.
Due to the need for $\omega$-wise (or ``pathwise'') definitions 
of such objects, we first prove
a preliminary result on Wiener Integrals.

\subsection{A property of the Wiener Integral}

In general, the stochastic integral is a random variable,
which does not a priori make sense pathwise: the integrator, in our case
Brownian Motion, is not of bounded variation, and therefore the stochastic
integral enjoys much weaker properties than a Riemann-Stieltjes-Integral. 

In this section we take a look at the stochastic integrals appearing in our analysis. First of all, note that the  integrands are deterministic, thus
providing a special case that is  called ``Wiener Integrals".
Moreover, due to the regularity of the integrands, which are of the form $\exp( - \lambda_j^s \tau )$, we are able to give
conditions on when the operator~$\frac{d}{ds}$ applied to the
Wiener Integral is well-defined as an $s$-dependent random variable.

\begin{lemma}\label{lemma:wienerintegral}
For any~$j\in\mathbb{N}$, let~$g_j:
{\mathscr{S}}\times[0,T]$.
Assume that, for any~$t\in[0,T]$,
the map~${\mathscr{S}}\ni s\mapsto g_j(t,s)$ is~$C^2$.

Let
$$ G_j(t,s):=\int_0^t g_j(\tau,s)\,dB_j(\tau)\qquad{\mbox{and}}\qquad
H_j(t,s):=\int_0^t \partial_s g_j(\tau,s)\,dB_j(\tau).$$
Assume that, for any~$j\in\mathbb{N}$ and
$s\in{\mathscr{S}}$,
\begin{equation}\label{DER2}| g_j(t,s)|+
|\partial_s g_j(t,s)|+|\partial_s^2 g_j(t,s)|\le C(s)\,\sqrt{\mu_j}\,\Gamma(t),\end{equation}
with~$C(s)\in(0,+\infty)$ for any fixed~$s$, 
\begin{equation}\label{DER3} M(s):=\sup_{\sigma\in(s/2,2s)} C(\sigma)<+\infty \qquad{\mbox{ for any fixed }}
s\in{\mathscr{S}},\end{equation}
and
\begin{equation}\label{DER4} \Gamma\in L^2([0,T],[0,+\infty)).\end{equation}
Then,
\begin{equation}\label{CA:01:001}
\partial_s\sum_{j=1}^{+\infty} G_j(t,s) \,e_j(x)=
\sum_{j=1}^{+\infty} \partial_s G_j(t,s) \,e_j(x)
=
\sum_{j=1}^{+\infty} H_j(t,s)\,e_j(x)
\end{equation}
as functions in~$L^2(\Omega,L^2(D\times[0,T]))$.
\end{lemma}

\begin{proof} 
First of all, we check that
\begin{equation}\label{CHA:CH}
\sum_{j=1}^{+\infty} G_j(t,s) \,e_j(x)\in L^2(\Omega,L^2(D\times[0,T])).
\end{equation}
To this end, we exploit~\eqref{DER2} and
\eqref{DER3} to see that
\begin{eqnarray*}
|G_j(t,s) |=
\left|\int_0^t g_j(\tau,s)\,dB_j(\tau) \right|
\le C(s)\,\sqrt{\mu_j}\,\int_0^t \Gamma(\tau)\,dB_j(\tau).
\end{eqnarray*}
Then, making use of~\eqref{DER4} together with
It\^o's Isometry, to see that, for every~$M\ge N\in{\mathbb{N}}$,
\begin{eqnarray*}
\left\| \sum_{j=N}^{M} G_j(t,s) \,e_j(x)\right\|^2_{L^2(\Omega,L^2(D\times[0,T]))}&=&\E\left[\int_0^T
\left\| \sum_{j=N}^{M} G_j(t,s) \,e_j(x)\right\|_{L^2(D)}^2
\,dt\right]\\
&=&\E\left[\int_0^T \sum_{j=N}^{M} |G_j(t,s) |^2 \,dt\right]\\&\le&C^2(s)\,\sum_{j=N}^{M}\mu_j\,\int_0^T\E\left[
\left| \int_0^t \Gamma(\tau)\,dB_j(\tau)\right|^2 \right]\,dt\\
&\le&C^2(s)\,\sum_{j=N}^{M}\mu_j\,\int_0^T\E\left[ \int_0^t \Gamma^2(\tau)\,d\tau \right]\,dt\\
&\le& C(s,T)\,\sum_{j=N}^{M}\mu_j.
\end{eqnarray*}
Notice that the latter quantity is infinitesimal for large~$N$ and~$M$,
thanks to~\eqref{CA5} and therefore the series in \eqref{CHA:CH}
produces a Cauchy sequence, thus proving~\eqref{CHA:CH}.

Fixed~$j\in\mathbb{N}$, $t\in[0,T]$, $s\in{\mathscr{S}}$
and~$h\in(-1,1)$ (to be taken sufficiently
small), we notice that
\begin{eqnarray*}
&&|G_j(t,s+h)-G_j(t,s)-hH_j(t,s)|\\
&=&\left| \int_0^t \Big( g_j(\tau,s+h)- g_j(\tau,s)-h\partial_sg_j(\tau,s)
\Big)\,dB_j(\tau)
\right|\\
&=&\left| \int_0^t \left( 
\int_0^ h \partial_sg_j(\tau,s+\sigma)\,d\sigma
-h\partial_sg_j(\tau,s)
\right)\,dB_j(\tau)
\right|\\
&=&\left| \int_0^t \left( 
\int_0^ h \Big(\partial_sg_j(\tau,s+\sigma)-\partial_sg_j(\tau,s)\Big)\,d\sigma
\right)\,dB_j(\tau)
\right|
\\
&=&\left| \int_0^t \left( 
\int_0^ h \left(
\int_0^\sigma
\partial_s^2 g_j(\tau,s+\rho)\,d\rho\right)\,d\sigma
\right)\,dB_j(\tau)
\right|\\&\le&\sqrt{\mu_j}\,
\int_0^t \left( 
\int_0^ h \left(
\int_0^\sigma
C(s+\rho)\,\Gamma(\tau)\,d\rho\right)\,d\sigma
\right)\,dB_j(\tau)
\\ &\le& M(s)\,\sqrt{\mu_j}\,\int_0^t \left( 
\int_0^ h \left(
\int_0^\sigma
\Gamma(\tau)\,d\rho\right)\,d\sigma
\right)\,dB_j(\tau)\\&=&
\frac{M(s)\,\sqrt{\mu_j}\,h^2}{2}\,\int_0^t 
\Gamma(\tau)\,dB_j(\tau),
\end{eqnarray*}
as long as~$h$ is small enough,
thanks to~\eqref{DER2} and~\eqref{DER3}. As a consequence,
\begin{eqnarray*}
\left| \frac{G_j(t,s+h)-G_j(t,s)}{h}-H_j(t,s)\right|^2
&\le& \frac{M^2(s)\,\mu_j\,h^2}{4}\,\left(\int_0^t 
\Gamma(\tau)\,dB_j(\tau)\right)^2.
\end{eqnarray*}
This, It\^o's Isometry and~\eqref{DER4} lead to
\begin{eqnarray*}
\E\left[\left| \frac{G_j(t,s+h)-G_j(t,s)}{h}-H_j(t,s)\right|^2\right]
&\le& \frac{M^2(s)\, \mu_j \,h^2}{4}\;\E\left[ \left(\int_0^t  \Gamma(\tau)\,dB_j(\tau)\right)^2\right]\\
&=&\frac{M^2(s)\,\mu_j\,h^2}{4}\;\E\left[ \int_0^t  \Gamma^2(\tau)\,d\tau\right]\\
&\le& C(s,T)\,\mu_j\,h^2,
\end{eqnarray*}
for some~$C(s,T)>0$. This estimate and Fatou's Lemma give that
\begin{equation*}
\E\left[\left| \partial_s G_j(t,s)-H_j(t,s)\right|^2\right]
\le
\lim_{h\to0}\E\left[\left| \frac{G_j(t,s+h)-G_j(t,s)}{h}-H_j(t,s)\right|^2\right]\le0.
\end{equation*}
Consequently, for any~$N\in{\mathbb{N}}$,
\begin{equation*}
\E\left[\int_0^T \sum_{j=1}^N\left| \partial_s G_j(t,s)-H_j(t,s)\right|^2 \,dt\right]=0.
\end{equation*}
Thus, using again Fatou's Lemma,
\begin{eqnarray*}
0 &=&\lim_{N\to+\infty}\E\left[\int_0^T \sum_{j=1}^N\left| \partial_s G_j(t,s)-H_j(t,s)\right|^2 \,dt\right]\\
&\ge&\E\left[\int_0^T \sum_{j=1}^{+\infty}\left| \partial_s G_j(t,s)-H_j(t,s)\right|^2 \,dt\right]
\\ &=&\E\left[\int_0^T \left\|
\sum_{j=1}^{+\infty}\big(\partial_s G_j(t,s)-H_j(t,s)\big)\,e_j(x)\right\|^2_{L^2(D)} \,dt\right]
\\ &=&
\left\|
\sum_{j=1}^{+\infty}\big(\partial_s G_j(t,s)-H_j(t,s)\big)\,e_j(x)\right\|^2_{L^2(\Omega,D\times[0,T])}
,\end{eqnarray*}
and accordingly
\begin{equation}\label{CA:01:002}
\sum_{j=1}^{+\infty} \partial_s G_j(t,s) \,e_j(x)
=
\sum_{j=1}^{+\infty} H_j(t,s)\,e_j(x)
\end{equation}
in~$L^2(\Omega,D\times[0,T])$.

Now we observe that, for any~$N\le M\in{\mathbb{N}}$,
\begin{eqnarray*}
&&
\left\| \sum_{j=N}^{M} \frac{ G_j(t,s+h)-G_j(t,s)}{h} \,e_j(x)
\right\|^2_{L^2(D)}+
\left\| \sum_{j=N}^{M} H_j(t,s) \,e_j(x)\right\|^2_{L^2(D)}
\\&=&\sum_{j=N}^{M} \frac{ |G_j(t,s+h)-G_j(t,s)|^2}{h^2}+
\sum_{j=N}^{M} |H_j(t,s)|^2\\
&=&
\sum_{j=N}^{M} \frac{ 1}{h^2}
\left|
\int_0^t \Big( g_j(\tau,s+h)- g_j(\tau,s)\Big)\,dB_j(\tau)
\right|^2
+
\sum_{j=N}^{M} \left| \int_0^t \partial_s g_j(\tau,s)\,dB_j(\tau)\right|^2
\\
&\le& M^2(s)\,
\sum_{j=N}^{M} \mu_j\,
\left|
\int_0^t \Gamma(\tau)\,dB_j(\tau)
\right|^2
+C^2(s)\,
\sum_{j=N}^{M} \mu_j\,\left| \int_0^t \Gamma(\tau)\,dB_j(\tau)\right|^2,
\end{eqnarray*}
in view of~\eqref{DER2} and~\eqref{DER3}. Using It\^o's Isometry and~\eqref{DER4}
we thereby find that
\begin{eqnarray*}
&& \;\E\left[
\left\| \sum_{j=N}^{M} \frac{ G_j(t,s+h)-G_j(t,s)}{h} \,e_j(x) \right\|^2_{L^2(D)} + \left\| \sum_{j=N}^{M} H_j(t,s) \,e_j(x)\right\|^2_{L^2(D)}\right]
\\ &\le& M^2(s)\,\left\{\E\left[ \sum_{j=N}^{M} \mu_j \left| \int_0^t \Gamma(\tau)\,dB_j(\tau) \right|^2\right]
+\E\left[ \sum_{j=N}^{M} \mu_j\left| \int_0^t \Gamma(\tau)\,dB_j(\tau)\right|^2\right]\right\}\\
&=&
M^2(s)\,\left\{ \sum_{j=N}^{M}\mu_j\,\E\left[ \int_0^t \Gamma^2(\tau)\,d\tau
\right] +\sum_{j=N}^{M}\mu_j\,\E\left[ \int_0^t \Gamma^2(\tau)\,d\tau\right]\right\}\\
&\le& C(s,T)\,\sum_{j=N}^{M}\mu_j,
\end{eqnarray*}
for some~$C(s,T)>0$. Therefore, recalling~\eqref{CA5},
we have that for any~$\epsilon>0$ there exists~$N_{\epsilon,s,T}\in{\mathbb{N}}$
such that if~$M\ge N\ge N_{\epsilon,s,T}$, we have that
$$ 
\left\| \sum_{j=N}^{M} \frac{ G_j(t,s+h)-G_j(t,s)}{h} \,e_j(x)
\right\|^2_{L^2(\Omega,D\times[0,T])}+
\left\| \sum_{j=N}^{M} H_j(t,s) \,e_j(x)\right\|^2_{L^2(\Omega,D\times[0,T])}\le \epsilon,$$
and so, by Fatou's Lemma,
$$ 
\left\| \sum_{j=N_{\epsilon,s,T}+1}^{+\infty} \frac{ G_j(t,s+h)-G_j(t,s)}{h} \,e_j(x)
\right\|^2_{L^2(\Omega,D\times[0,T])}+
\left\| \sum_{j=N_{\epsilon,s,T}+1}^{+\infty} H_j(t,s) \,e_j(x)\right\|^2_{L^2(\Omega,D\times[0,T])}\le \epsilon.$$
As a consequence, using It\^o's Isometry once again,
\begin{eqnarray*}
&& \frac12\,\left\|
\sum_{j=1}^{+\infty} \frac{G_j(t,s+h) -G_j(t,s)}{h}\,e_j(x)-
\sum_{j=1}^{+\infty} H_j(t,s)\,e_j(x)
\right\|^2_{L^2(\Omega,D\times[0,T])}\\
&\le& \left\|
\sum_{j=1}^{N_{\epsilon,s,T}} \frac{G_j(t,s+h) -G_j(t,s)}{h}\,e_j(x)-
\sum_{j=1}^{N_{\epsilon,s,T}} H_j(t,s)\,e_j(x)
\right\|^2_{L^2(\Omega,D\times[0,T])}+\epsilon\\
&=&
\left\|
\sum_{j=1}^{N_{\epsilon,s,T}} 
\int_0^t \left(\frac{g_j(\tau,s+h)-g_j(\tau,s)}{h}
-\partial_s g_j(\tau,s)\right)\,dB_j(\tau)\,e_j(x)
\right\|^2_{L^2(\Omega,D\times[0,T])}+\epsilon\\
&=& \E\left[\int_0^T
\sum_{j=1}^{N_{\epsilon,s,T}} \left|
\int_0^t \left(\frac{g_j(\tau,s+h)-g_j(\tau,s)}{h}
-\partial_s g_j(\tau,s)\right)\,dB_j(\tau)
\right|^2\,dt\right]+\epsilon\\
&=& \sum_{j=1}^{N_{\epsilon,s,T}}
\int_0^T \E\left[ 
\int_0^t \left(\frac{g_j(\tau,s+h)-g_j(\tau,s)}{h}
-\partial_s g_j(\tau,s)\right)^2\,d\tau\right]
\,dt+\epsilon.
\end{eqnarray*}
Hence, by~\eqref{DER2}, \eqref{DER3} and~\eqref{DER4},
\begin{eqnarray*}
&& \frac12\,\left\|
\sum_{j=1}^{+\infty} \frac{G_j(t,s+h) -G_j(t,s)}{h}\,e_j(x)-
\sum_{j=1}^{+\infty} H_j(t,s)\,e_j(x)
\right\|^2_{L^2(\Omega,D\times[0,T])}\\
&\le& M^2(s)\,h^2\,\sum_{j=1}^{N_{\epsilon,s,T}}\mu_j
\int_0^T \E\left[ 
\int_0^t \Gamma^2(\tau)\,d\tau\right]
\,dt+\epsilon\\&\le&
C(s,T)\,h^2\sum_{j=1}^{N_{\epsilon,s,T}}\mu_j+\epsilon.
\end{eqnarray*}
This, \eqref{CHA:CH} and Fatou's Lemma yield that
\begin{eqnarray*}
\epsilon&=&\lim_{h\to0} C(s,T)\,h^2
\sum_{j=1}^{N_{\epsilon,s,T}}\mu_j+\epsilon
\\ 
&\ge& \frac12\,\lim_{h\to0}
\left\|
\sum_{j=1}^{+\infty} \frac{G_j(t,s+h) -G_j(t,s)}{h}\,e_j(x)-
\sum_{j=1}^{+\infty} H_j(t,s)\,e_j(x)
\right\|^2_{L^2(\Omega,D\times[0,T])}\\
&=&\frac12\,\lim_{h\to0}
\left\|\frac1{h}\left( \sum_{j=1}^{+\infty} G_j(t,s+h) \,e_j(x)
-\sum_{j=1}^{+\infty} G_j(t,s) \,e_j(x)\right)-
\sum_{j=1}^{+\infty} H_j(t,s)\,e_j(x)
\right\|^2_{L^2(\Omega,D\times[0,T])}\\
&\ge&\frac12\,
\left\|\lim_{h\to0}\frac1{h}\left( \sum_{j=1}^{+\infty} G_j(t,s+h) \,e_j(x)
-\sum_{j=1}^{+\infty} G_j(t,s) \,e_j(x)\right)-
\sum_{j=1}^{+\infty} H_j(t,s)\,e_j(x)
\right\|^2_{L^2(\Omega,D\times[0,T])}\\
&=&
\frac12\,
\left\|\partial_s\sum_{j=1}^{+\infty} G_j(t,s) \,e_j(x)
-\sum_{j=1}^{+\infty} H_j(t,s)\,e_j(x)
\right\|^2_{L^2(\Omega,D\times[0,T])}.
\end{eqnarray*}
Since~$\epsilon$ can be taken arbitrarily small, we thereby conclude that
$$ \partial_s\sum_{j=1}^{+\infty} G_j(t,s) \,e_j(x)
=\sum_{j=1}^{+\infty} H_j(t,s)\,e_j(x)
$$ 
in~$L^2(\Omega,D\times[0,T])$. This and~\eqref{CA:01:002}
give~\eqref{CA:01:001}, as desired.
\end{proof}

Next, we recall Lemma~2.2 of~\cite{sprekelsvaldinoci},
which is an auxiliary result on the derivatives of a function of exponential type. 

\begin{lemma}\label{lemma:efunkt}
Define for fixed $\lambda>0$ and $t>0$ the real-valued function
\begin{equation}\label{eq:efkt}
E_{\lambda,t}(s) := e^{- \lambda^s t} \qquad \mbox{ for } s > 0 .
\end{equation}
Then there exist constants $C_i>0$ such that, for all $\lambda>0$, $t\in (0,T]$ and 
$s>0$ ,
we have that
$$\left|E_{\lambda,t}(s)\right| \leq C_0$$
and
\begin{equation}\label{eq:sgmajorant}
\left|\frac{d^k}{ds^k}E_{\lambda,t}(s)\right|  \leq \frac{C_k}{s^k} 
\,(1 + |\ln (t) |^k),\qquad
{\mbox{for all $1 \leq k \leq 4$.}} 
\end{equation}
\end{lemma}

Using Lemmata \ref{lemma:wienerintegral} and \ref{lemma:efunkt}
we can now take into account
the first and second derivatives of the solutions
with respect to the fractional parameter~$s$,
according to the following result:

\begin{proposition}\label{prop:sdifferential}
Let $\mathcal{L}$ satisfy Assumption~\ref{ass:ew}
 and let~$Q$ satisfy~\eqref{CA5}.
Let the initial data $y_0 \in L^2(D)$ be deterministic. 

Then
\begin{equation}\label{eq:sderivatives}
 \partial_s y(s) = \sum_{j=1}^{+\infty} \partial_s y_j(\cdot , s) e_j 
 \quad \textup{ and } \quad 
  \partial_{ss}^2 y(s) = \sum_{j=1}^{+\infty} \partial_{ss} y_j(\cdot , s) e_j ,
\end{equation}
are functions in $L^2\left(\Omega , L^2(D\times [0,T])\right)$.
 
 Moreover, for a fixed $s \in (0, +\infty)$, the random variables $ \omega \mapsto \|\partial_s y(s, \omega)\|_{L^2(D \times [0,T])} $
 and $ \omega \mapsto \|\partial_{ss} y(s, \omega)\|_{L^2(D \times [0,T])} $ are almost surely finite.  
\end{proposition}
\begin{proof}
{F}rom~\eqref{eq:OU} and~\eqref{eq:efkt}, we know that
\begin{equation}\label{8ihfeds234YSHNNA}
y_j(t,s) = y_{j,0} E_{\lambda_j,t} + \sqrt{ \mu_j } \int_0^t E_{\lambda_j,t-\tau} dB_j(\tau),
\end{equation}
Now we exploit Lemma~\ref{lemma:wienerintegral},
used here with~$g_j:=\sqrt{\mu_j} E_{\lambda_j,t}$,
in the case of the first derivative
and~$g_j:=\sqrt{\mu_j} \frac{dE_{\lambda_j,t}}{ds}$
in the case of the second derivative:
in this setting, in light of~\eqref{eq:sgmajorant}
we can take~$C(s):=C\,\left(\frac1s+\frac{1}{s^4}\right)$, with~$C>0$,
and $\Gamma(t):=1+|\ln t|^4$ in
Lemma~\ref{lemma:wienerintegral},
and then assumptions~\eqref{DER2},
\eqref{DER3} and~\eqref{DER4} are satisfied.

Accordingly, from~\eqref{8ihfeds234YSHNNA} and
Lemma~\ref{lemma:wienerintegral}
we obtain~\eqref{eq:sderivatives}, as desired.

Then, by Lemma~\ref{BORELCANTE},
we conclude that the first and
second derivatives of the solution with respect to~$s$ are almost surely finite
in~$L^2(D\times[0,T])$.
\end{proof}

Note that for Proposition \ref{prop:sdifferential} the function $y(s)(x,t)$ is only required to be a $L^2(D)$-valued adapted stochastic process, as proved in Lemma \ref{lemma:existence}. The proof used only 
$L^2\left(\Omega, L^2(D \times [0,T])\right)$-norms, no additional $\mathcal{H}^s$-regularity is needed, therefore, Assumption \ref{ass:decay} is not needed
in Proposition \ref{prop:sdifferential}.

\subsection{Optimality conditions}\label{sec:optimality}

In this section, we establish first-order necessary conditions and sufficient optimality conditions of optimal controls. 

\begin{theorem}\label{theo:optcond}
Let $y_0 \in L^2(D)$ be deterministic, and let $y = y(s)$ be a solution to
the state equation \eqref{eq:stateeq} in the sense of  the $L^2(D)$-valued stochastic
process $y(s): \Omega \times [0,T] \to L^2(D)$ of Lemma \ref{lemma:existence}.  Then the following holds true for a fixed realisation $\omega \in \Omega$: 
 
 \textbf{(i) necessary condition:} If $\bar{s}
=\bar{s}(\omega) $ is an optimal parameter for (IP) and $y(\bar{s})$ the associated unique solution to the state system \eqref{eq:stateeq}, then for almost every $\omega \in \Omega$
 \begin{equation}\label{eq:necessary}
  \int_0^T \int_D (y(\bar{s}) - y_D ) \partial_s y(\bar{s}) \, dx dt  \; + \; \Phi' (\bar{s}) \; = \; 0 .
 \end{equation}
  \textbf{(ii) sufficient condition:} If $\bar{s}=\bar{s}(\omega) \in (0, L)$ satisfies
the necessary condition \eqref{eq:necessary}, and if in addition 
   \begin{equation}
  \int_0^T \int_D \left(\partial_sy(\bar{s})\right)^2  + (y(\bar{s}) - y_D ) \partial_{ss}^2 y(\bar{s}) \, dx dt  \; + \; \Phi'' (\bar{s}) \; > \; 0 
 \end{equation}
 for almost every $\omega \in \Omega$, then $\bar{s}$ is optimal for (IP). 
\end{theorem}

\begin{proof}
By Proposition \ref{prop:sdifferential}, the map
$$ s \mapsto \mathcal{J}(s) := \mathcal{J}(y(s),s)$$ is twice differentiable on $(0,
+\infty)$. 
 By the chain rule, 
 \begin{equation*}
 \begin{aligned}
  \mathcal{J}'(\bar{s}) &= \frac{d}{ds} \mathcal{J}(y(\bar{s}), \bar{s})
= \partial_y \mathcal{J}(y(\bar{s}), \bar{s}) \circ \partial_s y(\bar{s}) \; + \; \partial_s \mathcal{J}(y(\bar{s}), \bar{s}) \\
  &= \int_0^T \int_D (y(\bar{s}) - y_D ) \partial_s y(\bar{s}) dx dt  \; + \; \Phi' (\bar{s}),
   \end{aligned}
 \end{equation*}
and  assertion (i) follows.
Also, assertion (ii) is a consequence of the following computation:
 \begin{equation*}
 \begin{aligned}
  \mathcal{J}''(\bar{s}) &= \frac{d}{ds} \mathcal{J}(y(\bar{s}), \bar{s}) = \partial_y \mathcal{J}(y(\bar{s}), \bar{s}) \circ \partial_s y(\bar{s}) \; + \; \partial_s \mathcal{J}(y(\bar{s}), \bar{s}) \\
  &= \int_0^T \int_D (y(\bar{s}) - y_D ) \partial_s y(\bar{s}) dx dt  \; + \; \Phi' (\bar{s}).
 \end{aligned}
 \end{equation*}
The proof of Theorem~\ref{theo:optcond} is thus complete.
\end{proof}

\section{Existence of optimal controls}\label{sec:controls}

The existence of
pathwise optimal controls is shown by checking that,
for fixed $\omega \in \Omega$, there exists a subsequence $y(s_k)$ which strongly converges to the optimal $y$ in $L^2(D \times [0,T])$. 

To show the strong convergence, we use a compactness result, which
proves that under certain assumptions solutions enjoy
a suitable H\"older regularity in time which is independent of the fractional exponent. 

\begin{lemma}[Compactness lemma]\label{lemma:compactness}
Let the initial data $y_0$ be deterministic,
with $y_0 \in \mathcal{H}^{s/2}$. Let Assumptions \ref{ass:EXPL} and \ref{ass:decay} be satisfied. 

Then, for a fixed realisation $\omega \in \Omega$,
the sequence  $\left\lbrace y_{s_k}(\omega) \right\rbrace_{k \in \mathbb{N}}$
of solutions to the state equation  \eqref{eq:stateeq}
with initial datum $y_0$ contains a subsequence that converges strongly in $L^2(D \times [0,T])$.
\end{lemma}

\begin{proof}
Recall that  for solutions of \eqref{eq:stateeq} in the sense of Definition \ref{def:solution} we know 
 \begin{enumerate}
  \item from Proposition \ref{prop:l2l2hs} that for all $s_k \in \mathscr{S}$ and  almost every $\omega \in \Omega$,
  \begin{equation*}
    \sup_k \left(\|y_{s_k}(\omega)\|_{L^2([0,T], \mathcal{H}^{s_k})}\right) < +\infty,
   \end{equation*}
  \item from Proposition \ref{prop:l2l2l2-norm} that for all $s_k \in (0, L)$ and for almost every $\omega \in \Omega$, 
  \begin{equation*}
    \sup_k \left(\|y_{s_k}(\omega) \|_{L^2(D \times [0,T])} \right) < +\infty,
   \end{equation*}
     \item from Lemma \ref{lemma:holderpath} that the trajectories of the family of stochastic processes $y_{s_k}(t)$ are in $C^{\delta_k}([0,T], L^2(D))$ for every $k$
     and  $\delta_k \geq \delta_* \geq \delta_0 > 0$.
 \end{enumerate}
Therefore, we know that $y_{s_k}$ is  a sequence (in $k$) of $L^2(D)$-valued stochastic processes (in $(x,t)$) with $\delta$-H\"older continuous sample paths
and  $y_{s_k} (\omega) \in L^2([0,T], \mathcal{H}^{s_k})$ for fixed $\omega \in \Omega$. 
Notice that, by~(3),
\begin{eqnarray*}&&
C\ge\|y_{s_k}(t)\|_{L^2(D)}^2=\sum_{i=1}^{+\infty} | y_{s_k,i}(t)|^2\\ {\mbox{and }}&&
C\,|t-t'|^{\delta_k}\ge
\| y_{s_k}(t)-y_{s_k}(t')\|_{L^2(D)}^2
=
\sum_{i=1}^{+\infty} | y_{s_k,i}(t)-y_{s_k,i}(t') |^2,
\end{eqnarray*}
and so
the
infinite string $\left(\left\lbrace y_{s_k,1}\right\rbrace_{k \in \mathbb{N}},\left\lbrace y_{s_k,2}\right\rbrace_{k \in \mathbb{N}},\ldots \right)$ lies in the space 
  \begin{equation*}
   C^{\delta_0}([0,T]) \times C^{\delta_0}([0,T]) \times \ldots
   \end{equation*}
  Hence, there exists a subsequence denoted by $(s_k)_m$ which converges in this product space to an infinite string of the form $\left((y^*_s)_1, (y^*_s)_2, \ldots\right)$, and every $(y^*_s)_j \in C^{\delta_0}([0,T])$.
  We define 
  \begin{equation*}
   y^*(x,t) = \sum_{j \in \mathbb{N}}  y^*_j e_j(x).
  \end{equation*}
  The convergence of $ y_{(s_k)_m} \longrightarrow y^*$ follows exactly as in the compactness lemma in the deterministic case, which is Lemma 6.1. of \cite{sprekelsvaldinoci}, by using also~(1) and~(2).
  The details are therefore omitted. 
\end{proof}

\begin{theorem}\label{theo:controlexist}
Let the initial data $y_0$ be deterministic,
and let Assumptions \ref{ass:EXPL} 
and \ref{ass:decay} be satisfied.  Moreover, let the initial data  satisfy 
\begin{equation}\label{SJok}
\sup_{s \in \mathscr{S}} \|y_0\|_{\mathcal{H}^s} < +\infty.\end{equation}
Then 
for almost every fixed $\omega \in \Omega$, the functional~$\mathcal{J}(\omega)$ attains a
minimum in $\mathscr{S}^{\circ}$, and moreover
 \begin{equation*}
\inf_{s \in \mathscr{S}} \mathcal{J}(\omega)  < +\infty .
 \end{equation*}
\end{theorem}

\begin{proof}
Note first that, by our assumptions on  $\Phi(s)$, we can find $s^* \in \mathscr{S}^{\circ }$ such that
$\mathcal{J}(s^*,\omega) < +\infty$, and, in view of~\eqref{eq:phiassumption}, we infer that
\begin{equation*}
 0 < \inf_{ s \in \mathscr{S}^{\circ } } \mathcal{J}(s, \omega ) < +\infty, \qquad \textup{ for any fixed } \omega \in \Omega .
\end{equation*}
 We pick a minimizing sequence $\lbrace s_k \rbrace_{k \in \mathbb{N}}
\subset  \mathscr{S}^{\circ } $,
and consider for every $k\in \mathbb{N}$
the unique solution $y_k = y(s_k)$ to the state system \eqref{eq:stateeq}
with initial datum $y_0$. 
 Without loss of generality, we can assume 
 \begin{equation*}
 \mathcal{J}(s_k) \leq 1 + \mathcal{J}(s^*) \qquad \forall k \in \mathbb{N} \textup{ for fixed } \omega \in \Omega.
 \end{equation*}
This and~\eqref{eq:cost} give
the almost sure finiteness of 
 $\|y_k(\omega)  \|_{L^2(D \times  [0,T]) } $. 

In view of~\eqref{eq:phiassumption}, the minimizing sequence  $s_k$ is bounded and we may assume 
without loss of generality 
that~$s_k \to \bar{s}$ for some $\bar{s} \in  \mathscr{S}^{\circ } $. 

Recalling~\eqref{SJok}
and Proposition~\ref{prop:l2l2l2-norm}, we can apply the compactness result
in Lemma \ref{lemma:compactness}, with $\delta_0 = \frac{1}{4}$, and
select a (not relabeled) subsequence
such that  $\left\lbrace y_k\right\rbrace_{k \in \mathbb{N}}$  converges strongly in $L^2(D \times [0,T])$ for fixed $\omega$ to a limit $\bar{y}$. 
Then, thanks to\footnote{As a side remark,
we note that the $\omega$-wise identification $\bar{y}(\omega) =
y(\bar{s}, \omega)$ is not enough to ensure that the 
optimal $ y(\bar{s})$ found in
Theorem~\ref{theo:controlexist}
is an adapted stochastic process, since~$\mathbb{P}$-measurability may
be lost when passing to the limit. Therefore, it remains an open problem
to show that  $ y(\bar{s})$ as a function of $(\omega, x, t)$ is a
solution to \eqref{eq:stateeq} in the sense of Definition \ref{def:solution}. } 
the uniqueness of solutions to the deterministic optimization problem,  which is Theorem~4.2 of \cite{sprekelsvaldinoci}, the identification $\bar{y}(\omega) = y(\bar{s}, \omega)$ is meaningful at the level of fixed $\omega$. 
\end{proof}

\begin{appendix}

\section{An auxiliary result of Borel-Cantelli type}

We state here a simple consequence of the Borel-Cantelli Lemma,
which is used several times in the proofs of the main results.

\begin{lemma}\label{BORELCANTE}
Let~$Z$ be a Banach space, with norm~$\|\cdot\|_Z$, and~$z:\Omega\to Z$.
Assume that
\begin{equation}\label{IAMla}
\|z\|_{L^2(\Omega,Z)}<+\infty.
\end{equation}
Then, the random variable
$$ \Omega\ni\omega\mapsto \|z(\omega)\|_Z$$
is almost surely finite.
\end{lemma}

\begin{proof} For any~$m\in{\mathbb{N}}\cup\{+\infty\}$, we define
$$ A_m:=\big\{\omega\in\Omega {\mbox{ s.t. }}
\|z(\omega)\|_{Z}^2\ge 2^m\big\}.$$
{F}rom~\eqref{IAMla} and the
Chebychev's inequality, we see that
\begin{eqnarray*}
\mathbb{P}(A_m)\le\frac{1}{2^m} \,\E \left[\|z\|_{Z}^2 \right]=\frac{1}{2^m} \,\|z\|_{L^2(\Omega,Z)}^2,
\end{eqnarray*}
and therefore
$$ \sum_{m=0}^{+\infty} \mathbb{P}(A_m)<+\infty.$$
{F}rom this and the Borel-Cantelli Lemma, we conclude that
$$ 0=\mathbb{P}(A_\infty)=\mathbb{P}\left(
\big\{\omega\in\Omega {\mbox{ s.t. }}
\|z(\omega)\|_{Z}^2=+\infty\big\}
\right),$$
which leads to the desired result.
\end{proof}

\end{appendix}

\section*{Acknowledgements} 
We are very much indebted to the Referee for her or his very
fruitful comments.
Research of Sections 1-3 was supported by
the Russian Science Foundation grant Nr. 14-21-00035. The second author
was supported by the Australian Research Council Discovery Project
DP170104880 NEW Nonlocal Equations at Work. The first author
thanks Nina
Gantert and Mauro Rosestolato for enlightening discussions on pathwise 
stochastic integration.

\bibliographystyle{abbrv}   


\begin{thebibliography}{10}

\bibitem{ABA}
N.~Abatangelo, E.~Valdinoci.
\newblock Getting acquainted with the fractional Laplacian.
\newblock {\em Springer INdAM Ser.}, Springer, Cham, 2019.

\bibitem{antil:15}
H.~Antil, E.~Ot\'arola, and A.~J. Salgado.
\newblock Some applications of weighted norm inequalities to the error analysis
  of pde constrained optimization problems.
\newblock {\em IMA Journal of Numerical Analysis}, 38(2):852--883,
  2018.

\bibitem{antil:14}
H.~Antil and E.~Ot\'arola.
\newblock A fem for an optimal control problem of fractional powers of elliptic
  operators.
\newblock {\em SIAM Journal on Control and Optimization}, 53(6):3432--3456,
  2015.

\bibitem{antil:15b}
H.~Antil, E.~Ot\'arola, and A.~J. Salgado.
\newblock A space-time fractional optimal control problem: Analysis and
  discretization.
\newblock {\em SIAM Journal on Control and Optimization}, 54(3):1295--1328,
  2016.

\bibitem{benner2016block}
P.~Benner, A.~Onwunta, and M.~Stoll.
\newblock Block-diagonal preconditioning for optimal control problems
  constrained by pdes with uncertain inputs.
\newblock {\em SIAM Journal on Matrix Analysis and Applications},
  37(2):491--518, 2016.

\bibitem{bors:15}
D.~Bors.
\newblock Optimal control of nonlinear systems governed by dirichlet fractional
  Laplacian in the minimax framework.
\newblock Preprint, \url{http://arxiv.org/abs/1509.01283}, 2015.

\bibitem{quarteroni}
P.~Chen, A.~Quarteroni, and G.~Rozza.
\newblock Stochastic optimal Robin boundary control problems of
  advection-dominated elliptic equations.
\newblock {\em SIAM Journal on Numerical Analysis}, 51(5):2700--2722, 2013.

\bibitem{dpz}
G.~Da~Prato and J.~Zabczyk.
\newblock {\em Stochastic equations in infinite dimensions}, volume 152 of {\em
  Encyclopedia of Mathematics and its Applications}.
\newblock Cambridge University Press, Cambridge, second edition, 2014.

\bibitem{dentcheva2006portfolio}
D.~Dentcheva and A.~Ruszczy{\'n}ski.
\newblock Portfolio optimization with stochastic dominance constraints.
\newblock {\em Journal of Banking \& Finance}, 30(2):433--451, 2006.

\bibitem{gunzburger}
M.~D. Gunzburger, H.-C. Lee, and J.~Lee.
\newblock Error estimates of stochastic optimal {N}eumann boundary control
  problems.
\newblock {\em SIAM J. Numer. Anal.}, 49(4):1532--1552, 2011.

\bibitem{manouzi}
L.~S. Hou, J.~Lee, and H.~Manouzi.
\newblock Finite element approximations of stochastic optimal control problems
  constrained by stochastic elliptic {PDE}s.
\newblock {\em J. Math. Anal. Appl.}, 384(1):87--103, 2011.

\bibitem{albatross}
N.~E. Humphries, H.~Weimerskirch, N.~Queiroz, E.~J. Southall, and D.~W. Sims.
\newblock Foraging success of biological l{\'e}vy flights recorded in situ.
\newblock {\em Proceedings of the National Academy of Sciences},
  109(19):7169--7174, 2012.

\bibitem{kouri2013}
D.~P. Kouri, M.~Heinkenschloss, D.~Ridzal, and B.~G. van Bloemen~Waanders.
\newblock A trust-region algorithm with adaptive stochastic collocation for
  {PDE} optimization under uncertainty.
\newblock {\em SIAM J. Sci. Comput.}, 35(4):A1847--A1879, 2013.

\bibitem{kruse}
R.~Kruse and S.~Larsson.
\newblock Optimal regularity for semilinear stochastic partial differential
  equations with multiplicative noise.
\newblock {\em Electron. J. Probab.}, 17:19 pp., 2012.

\bibitem{LeGall}
J.-F. Le~Gall.
\newblock {\em Brownian motion, martingales, and stochastic calculus}, volume
  274 of {\em Graduate Texts in Mathematics}.
\newblock Springer, french edition, 2016.

\bibitem{lord}
G.~J. Lord, C.~E. Powell, and T.~Shardlow.
\newblock {\em An introduction to computational stochastic {PDE}s}.
\newblock Cambridge Texts in Applied Mathematics. Cambridge University Press,
  New York, 2014.

\bibitem{markowitzbook}
H.~M. Markowitz and G.~P. Todd.
\newblock {\em Mean-variance analysis in portfolio choice and capital markets},
  volume~66.
\newblock John Wiley \& Sons, 2000.

\bibitem{duncan}
B.~Pasik-Duncan.
\newblock Asymptotic distribution of some quadratic functionals of linear
  stochastic evolution systems.
\newblock {\em Journal of Optimization Theory and Applications},
  75(2):389--400, 1992.
  
\bibitem{PERT}
B. Perthame, E. Ribesy, D. Salort.
\newblock Career plans and wage structures: a mean field game approach.
\newblock {\em Mathematics in Engineering},
1(1): doi:10.3934/Mine.2018.1.47, 2019.

\bibitem{econpaper2}
W.~Price, A.~Martel, and K.~Lewis.
\newblock A review of mathematical models in human resource planning.
\newblock {\em Omega}, 8(6):639 -- 645, 1980.

\bibitem{rosseel2012optimal}
E.~Rosseel and G.~N. Wells.
\newblock Optimal control with stochastic pde constraints and uncertain
  controls.
\newblock {\em Computer Methods in Applied Mechanics and Engineering},
  213:152--167, 2012.

\bibitem{SV-spectrum}
R.~Servadei and E.~Valdinoci.
\newblock On the spectrum of two different fractional operators.
\newblock {\em Proceedings of the Royal Society of Edinburgh: Section A
  Mathematics}, 144(4):831–855, 2014.

\bibitem{OMVpolymerSiebererSPE}
M.~Sieberer, T.~Clemens, J.~Peisker, and S.~Ofori.
\newblock Polymer flood field implementation - pattern configuration and
  horizontal versus vertical wells.
   {\em SPE Improved Oil Recovery Conference Paper}, 01 2018.

\bibitem{sprekelsvaldinoci}
J.~Sprekels and E.~Valdinoci.
\newblock A new type of identification problems: Optimizing the fractional
  order in a nonlocal evolution equation.
\newblock {\em SIAM Journal on Control and Optimization}, 55(1):70--93, 2017.

\bibitem{econpaper1}
B.~Stewart, D.~Webster, S.~Ahmad, and J.~Matson.
\newblock Mathematical models for developing a flexible workforce.
\newblock {\em International Journal of Production Economics}, 36(3):243 --
  254, 1994.

\bibitem{Taware2017}
S.~Taware, A.~H. Alhuthali, M.~Sharma, and A.~Datta-Gupta.
\newblock Optimal rate control under geologic uncertainty: water flood and eor
  processes.
\newblock {\em Optimization and Engineering}, 18(1):63--86, Mar 2017.


\bibitem{tiesler2012}
H.~Tiesler, R.~M. Kirby, D.~Xiu, and T.~Preusser.
\newblock Stochastic collocation for optimal control problems with stochastic
  {PDE} constraints.
\newblock {\em SIAM J. Control Optim.}, 50(5):2659--2682, 2012.

\end{thebibliography}

\def\cprime{$'$} \def\cprime{$'$}

\end{document}